\documentclass[12pt]{article}
\usepackage{amsmath,amsthm,amsfonts,amssymb,times,xcolor}
\usepackage{url, hyperref}
\usepackage{marginnote}

\newtheorem{theorem}{Theorem}[section]

\newtheorem{lemma}[theorem]{Lemma}
\newtheorem{proposition}[theorem]{Proposition}

\def\currentbound{8896}

\allowdisplaybreaks

\begin{document}

\title{Only finitely many $s$-Cullen numbers are repunits for a fixed $s \ge 2$}

\author{Michael Filaseta \\
\normalsize Department of Mathematics,
University of South Carolina \\
\normalsize Columbia, SC 29208,
E-mail:  filaseta@math.sc.edu
\and
Jon Grantham \\
\normalsize Institute for Defense Analyses,
Center for Computing Sciences \\ 
\normalsize Bowie, Maryland 20715, 
E-mail: grantham@super.org
\and
Hester Graves \\
\normalsize Institute for Defense Analyses,
Center for Computing Sciences \\ 
\normalsize Bowie, Maryland 20715, 
E-mail:  hkgrave@super.org}

\maketitle 

\abstract{We show that for any integer $s \geq 2$, there are only finitely many $s$-Cullen numbers that are repunits.  
More precisely, for fixed $s \ge 2$, there are only finitely many integers $n$, $b$, and $q$ with $n \geq 2$, $b \geq 2$ and $q \geq 3$ such that \[C_{n,s} = ns^n + 1 = \frac{b^q -1}{b-1}.\] 
The proof is elementary and effective, and it is used to
show that there are no $s$-Cullen repunits, other than explicitly known ones, for all $s \in [2,\currentbound]$.}

\vskip 10pt\noindent

\thispagestyle{empty}

\section{Introduction}

An \textit{$s$-Cullen number} is a number of the form $C_{n,s} = ns^n + 1$, where $s$ and $n$ are positive integers with $s \geq 2$. 
A \textit{repunit in an integer base $b \ge 2$} is a positive integer that we can write as 
\[
\frac{b^q -1}{b-1} = \sum_{j=0}^{q-1} b^j = b^{q-1} + b^{q-2} + \cdots + b + 1
 = (1 1 \cdots 1)_{b},
 \]
for some  positive integer $q$.  
We will refer to a positive integer $m$ as being a \textit{repunit} if there exists a base $b \ge 2$ for which $m$ is a repunit in base $b$
and the number of digits of $m$ in base $b$ is at least $3$.  This somewhat awkward digit requirement is motivated
by the fact that every integer $m \ge 3$ is a $2$ digit repunit in base $m-1$.  

In a previous paper \cite{ccr}, the second and third authors assumed a weak form of the $abc$-conjecture to show that 
the collection of all $s$-Cullen numbers as $s$ varies contains only finitely many repunits, outside of two exceptional families, 
which we describe next.  
The number $C_{1,s} = s+1$ is a repunit in a base $b \ge 2$ if and only if $s$ is one less than a repunit in a base $b$.
A more interesting example occurs when $s^{2}$ is a triangular number $b(b+1)/2$.  In this case, 
\[
C_{2,s}=2s^2+1=2\frac{b(b+1)}2+1=b^2+b+1
\]
is a repunit in base $b$ with three digits.  
Euler \cite{euler} was the first to show that 
there are infinitely many examples of such $s$, the smallest of which are $1$, $6$, $35$, $204$ and $1189$. 
A proof can be obtained by looking at solutions to the Pell equation $x^{2} - 8y^{2} = 1$ with $x$ odd.

Here, we provide an unconditional and effective proof of the following.

\begin{theorem}\label{themaintheorem}
For each integer $s \ge 2$, there are at most finitely many $s$-Cullen numbers $C_{n,s}$ that are repunits.
\end{theorem}

After establishing Theorem~\ref{themaintheorem}, we demonstrate computationally that for all integers $s \in [2, \currentbound] $, 
there are no $s$-Cullen numbers which are repunits other than those described before Theorem~\ref{themaintheorem}.
In general, for a fixed $s\ge 2$, we are interested in integers $b \ge 2$, $q \ge 3$ and $n \ge 2$ satisfying
\begin{equation} \label{eq:1}
n s^{n} = b (b^{q-2} + \cdots  + b + 1). 
\end{equation}
In the next two sections, we will state an explicit version of Theorem~\ref{themaintheorem}; in particular, we supply an explicit upper bound
on $n$ for which \eqref{eq:1} holds for a fixed $s \ge 2$.

\section{Upper bounds on $n$ and $q$}

Let $s$, $b$, $q$ and $n$ be integers satisfying
$s \ge 2$, $b \ge 2$, $q \ge 3$ and $n \ge 1$.  
Set
\[
q_s = \max\{ 7,s \}, \quad
c_s = 1+\log_2(s) \quad \text{ and } \quad
N_s = \lfloor \max\{c_s^{2},127\}\rfloor.
\]
We also make use of the notation $\Phi_{m}(x)$ for the $m$-th cyclotomic polynomial.  
Our aim in this section is to establish the following.

\begin{theorem}\label{lnlq}
Let $s$, $b$, $q$ and $n$ be integers satisfying the conditions
$s \ge 2$, $b \ge 2$, $q \ge 3$, $n \ge 1$ and  \eqref{eq:1}.  
Then either $q \le q_{s}$ or $n \le N_{s}$.
\end{theorem}

\subsection{Some preliminaries}

\begin{lemma}\label{lemma1}
Under the conditions of Theorem~\ref{lnlq}, we have
\[
c_sn = (1 + \log_2(s)) n > q.
\]
\end{lemma}

\begin{proof}
Since $b \ge 2$, $q \ge 3$ and \eqref{eq:1} hold, we see that 
$2ns^n > 2b^{q-1} \ge 2^{q}$.  
One checks that $x - \log_{2}(x)$ is increasing for $x \ge 2$ to deduce that $n \ge \log_{2}(n) + 1$ for all positive integers $n$.  
Thus,
\[
2^{c_sn } = 2^{(1 + \log_2(s)) n} \ge 2^{\log_{2}(n) + 1 + \log_2(s) n} = 2 n s^{n} > 2^{q}.
\]
The lemma follows.
\end{proof}

\begin{lemma}\label{lemma2}
Under the conditions of Theorem~\ref{lnlq}, we have $b^q > s^n$.
\end{lemma}

\begin{proof}
The result follows from $b^q > b (b^{q-2} + \cdots  + b + 1)= n s^n \ge s^n$.
\end{proof}

The next result is well-known.  For a proof, one can see the corollary in Section~5 of \cite{dickson}
and the reference for Lemma~3 in \cite{tv}.

\begin{proposition}\label{cyclo} 
Let $q \ge 4$ and $b \ge 2$ be integers, and let $p$ be a prime.  
If $p | \Phi_{q-1}(b)$, then either $p \equiv 1 \pmod{q-1}$ or $p$ is the largest prime divisor of $q-1$ and $p^2 \nmid \Phi_{q-1}(b)$.  
All primes dividing $\Phi_{q-1}(b)$ are $\geq q$, with at most one exception, counting multiplicity.
\end{proposition}

For the next result, see Theorem~5 in \cite{tv}.

\begin{proposition}\label{roitman}
Let $q \ge 3$ and $b \ge 2$ be integers.  
We have $\Phi_{q-1}(b)\ge b^{\phi(q-1)}/2$.
\end{proposition}

\begin{lemma}\label{bound_on_phi} 
Let $q \ge 7$ be an integer.   
If $q \not\in \{ 8, 10, 12, 18, 30, 42, 60 \}$, then $\phi(q) > q/(2.61 \log \log q)$.
\end{lemma}

\begin{proof}  
Rosser and Schoenfeld \cite{rs} showed that 
\begin{equation*}
\phi(q) > \frac{q}{e^{\gamma}  \log \log q + \frac{2.5}{\log \log q}},
\end{equation*}
for all integers $q\geq 3$ with $q \neq 223092870$, where
$\gamma = 0.5772156649\ldots$ is the Euler–Mascheroni constant.
Hence, for $q \ge 3$ with $q \neq 223092870$, we want to show that
\[
\frac{q}{e^{\gamma}  \log \log q + \frac{2.5}{\log \log q}}  > \frac{q}{2.61 \log \log q} 
\]
which is equivalent to
\[
2.61 \log \log q > e^{\gamma}  \log \log q + \frac{2.5}{\log \log q}.
\]
This inequality holds provided
\[
 \log \log q \ge  \frac{3.02}{\log \log q},
 \]
 which in turn holds if $\log \log q \ge 1.74$ or, equivalently, if $q \geq 299$.  
Now, one can verify the lemma by computing $\phi(q)$ and $q/(2.61 \log \log q)$ directly for  $q = 223092870$ and all integers $q \in [7, 298]$.
 \end{proof}

\begin{proposition}\label{loglog} If $q$ is an integer with $q \ge 8$, then 
\begin{equation*}
\phi(q-1) > \frac{q}{3.4504 \log \log q}.
\end{equation*}
\end{proposition}

\begin{proof}
Let $q$ be an integer with $q \ge 8$.  
We know by Lemma \ref{bound_on_phi} that 
\[
\phi(q-1) > \frac{q-1}{2.61 \log \log (q-1)} ,
\quad \text{ for } q \ge 62.
\]
We can also see that when $q \geq 62$, we have 
\begin{equation*}
\frac{ \frac{62}{61} (q-1)} {\log \log (q-1)} \geq \frac{\frac{q}{q-1} (q-1)}{\log \log (q-1)} = \frac{q}{\log \log (q-1)} > \frac{q}{\log \log q},
\end{equation*}
thus revealing that 
\begin{equation*}
\phi(q-1) > 
\frac{q-1}{2.61 \log \log (q-1)} \geq \frac{q}{\left( \frac{62}{61} \right) 2.61 \log \log q} > \frac{q}{3.4504 \log \log (q)}.
\end{equation*}
A computation for $q \in [8, 61]$ completes the proof.
\end{proof}

\subsection{A proof of Theorem~\ref{lnlq}}

Let $s$, $b$, $q$ and $n$ be integers satisfying
$s \ge 2$, $b \ge 2$, $q > q_{s}$ and $n > N_{s}$. 
We want to prove that \eqref{eq:1} does not hold.  
Assume for some such $s$, $b$, $q$ and $n$ that \eqref{eq:1} does hold.  

Since $q > q_{s} = \max\{ 7,s \}$,  
Proposition \ref{cyclo} implies that 
\[
\frac{\Phi_{q-1}(b)}{\gcd(\Phi_{q-1}(b),q-1)} 
\] 
has all its prime factors $> s$ and is therefore coprime to $s$. 
Note that $\Phi_{q-1}(b)$ divides the right side of \eqref{eq:1}, and thus $ns^n$. 
So $\Phi_{q-1}(b)/\gcd(\Phi_{q-1}(b),q-1)$ must be a factor of $n$, and we have
\[
n\geq \frac{\Phi_{q-1}(b)}{\gcd(\Phi_{q-1}(b),q-1)} \geq \frac{\Phi_{q-1}(b)}{q-1}.
\]
By Propositions \ref{roitman} and \ref{loglog}, we have
\[
n \ge \frac{b^{\phi(q-1)}}{2(q-1)} > \frac{b^{\frac{q}{3.4504\log\log{q}}}}{2(q-1)}.
\]
Applying Lemma \ref{lemma2} and then Lemma \ref{lemma1} , we deduce
\[
n > \frac{s^{\frac{n}{3.4504\log\log{q}}}}{2(q-1)} > \frac{s^{\frac{n}{3.4504\log\log(c_s n)}}}{2c_s n}.
\]
Taking logarithms of both sides, with some rewriting, we deduce
\[
n \log(s) < 3.4504 \,\big(2 \log n + \log(2c_{s})\big) \log\log(c_{s}n).
\]
Since $n > N_{s}$, we have $n > c_{s}^{2}$.  Hence,
\begin{align*}
n \log(s) &< 3.4504 \,(2 \log n + \log(2n^{1/2})) \log\log(n^{3/2}) \\[5pt]
&= 3.4504 \,(2.5 \log n + \log 2) \big(\log\log n + \log (3/2) \big).
\end{align*}
The functions $x/\log x$ and $x/\log^{2}x$ are increasing functions for $x > e$ and for $x > e^{2} = 7.389\ldots$, respectively,
which implies
\[
\log x <
\begin{cases}
x/26.2169 &\text{for } x \ge 127 \\[5pt]
x/3.0702 &\text{for } x \ge \log(127)
\end{cases}
\]
and
\[
\log^{2} x < x/5.412, \quad \text{ for } x \ge 127.
\]
Since $n > N_{s} \ge 127$, we obtain
\[
\log n < \dfrac{n}{26.2169}, \quad \log\log n < \dfrac{\log n}{3.0702}, \quad \text{ and } \quad
\log^{2}n < \dfrac{n}{5.412}.
\]
Thus,
\begin{align*}
n \log(s) &< 3.4504 \,(2.5 \log n + \log 2) \bigg(\dfrac{\log n}{3.0702} + \log (3/2) \bigg) \\[5pt]
&< 2.8096 \log^{2} n + 4.2766 \log n + 0.9698 \\[5pt]
&< 2.8096 \log^{2} n + 4.2766 \log n +  \dfrac{0.9698 \log n}{\log 127} \\[5pt]
&< 2.8096 \log^{2} n + 4.4768 \log n \\[5pt]
&< 0.5192 \,n + 0.1708 \,n = 0.69 \,n < n \log(2).
\end{align*}
As $s \ge 2$, this is a contradiction, finishing the proof.

\section{Large $n$ and small $q$}

In this section, we establish the following.

\begin{theorem}\label{theorem3}
For each integer $s \ge 2$, there exists a positive integer $\mathcal{N}_s$ such that there are no solutions to \eqref{eq:1} 
in integers $n$, $q$ and $b$ satisfying $n > \mathcal{N}_s$, $3 \le q \le q_s$ and $b \ge 2$.
\end{theorem}

Before giving the proof, we note that $\mathcal{N}_s$ can be made explicit.  
For $s \ge 2$ fixed, we will choose $\mathcal{N}_s$ so that $n > \mathcal{N}_s$ implies
\begin{equation}\label{nsbound}
2^{1/n} n^{q_s/n} \le 1+ \dfrac{1}{s}.
\end{equation}
As the left-hand side decreases to $1$ for $n \ge 3$, such an $\mathcal{N}_s$ exists and is easily computed.
In the argument below, we will also want $n^{1/n} \le 1 + (1/s)$ for $n > \mathcal{N}_s$, which we note is implied by \eqref{nsbound}.

\begin{proof}[Proof of Theorem~\ref{theorem3}]
Let 
\begin{equation}\label{definitionfq}
f_q(b)=b^{q-2}+ b^{q-3}+ \cdots+1. 
\end{equation}
Then \eqref{eq:1} can be written as $ns^n=bf_q(b)$. 
Since $\gcd(b,f_{q}(b))=1$, we can write $s=uv$, where $\gcd(u,v) = 1$, $u=\gcd(b,s)$ and $v=\gcd(f_q(b),s)$.  
Let $k$ be such that $b=ku^n$. Note that $k$ is an integer dividing $n$ so that $1 \le k \le n$.  

Observe that $s=uv \ge 2$ and $\gcd(u,v) = 1$
imply that $u^{q-2} \ne v$. We consider two cases,
depending on whether $u^{q-2} >v$ or $u^{q-2} < v$.  

First, we assume that $u^{q-2} > v$.  
For any $q$, we have that 
\[
f_q(b) = \frac{ns^n}{b} = \frac{ns^n}{k u^n} = \frac{n}{k} v^n \leq nv^n.
\]
We also have that 
\[
f_q(b)\ge f_q(u^n)>u^{(q-2)n}.
\]  
Thus, $n > (u^{q-2}/v)^n$.  
Because  $u^{q-2}\ge v+1$, we have
\[
n > \bigg( 1+ \dfrac{1}{v} \bigg)^n \ge \bigg( 1+ \dfrac{1}{s} \bigg)^n.
\] 
We choose $\mathcal{N}_s$ large enough so that this cannot hold for $n > \mathcal{N}_{s}$.
Therefore, we obtain a contradiction in the case that $u^{q-2} > v$.

Now, assume that $1 \le u^{q-2} < v$. 
We have that 
\[
v^n\le f_q(b) = f_q(ku^n) \le k^{q-2} f_{q}(u^n)
\le n^{q} f_{q}(u^n) < 2 \,n^{q} u^{(q-2)n}.
\] 
Since our assumption is now that $1 \le u^{q-2} < v$, we deduce that
\[
2n^{q_{s}} \ge 2n^q > \bigg( \dfrac{v}{u^{q-2}} \bigg)^{n} \ge \bigg( \dfrac{v}{v-1} \bigg)^{n} = \bigg(1 + \dfrac{1}{v-1} \bigg)^{n} > \bigg( 1 + \dfrac{1}{s} \bigg)^{n}.
\]
We choose $\mathcal{N}_s$ large enough so that \eqref{nsbound} holds,
giving a contradiction and finishing the proof.
\end{proof}

Observe that Theorem~\ref{lnlq} and Theorem~\ref{theorem3} imply Theorem~\ref{themaintheorem} as follows.
Let $M_{s} = \max\{ N_{s}, \mathcal{N}_{s} \}$.  Then Theorem~\ref{lnlq} and Theorem~\ref{theorem3} imply that
for fixed $s \ge 2$ and all integers $n > M_{s}$, there are no integers $b \ge 2$ and $q \ge 3$ satisfying \eqref{eq:1}.  
Thus, there are $\le M_{s}$ different $s$-Cullen numbers which are repunits. 

\section{Computations}

In this section, we describe computations we did to obtain the following result. 

\begin{theorem}\label{theorem4}
For each integer $s \in [2,\currentbound]$, there are no $s$-Cullen numbers which are repunits other than those described before Theorem~\ref{themaintheorem}.
\end{theorem}

There were no particularly interesting examples that arose in our computations described below.  Although, for example,
we encountered 
$C_{158,313} = 158 \cdot 313^{158}+1$
which in base $2$ has its $19$ right-most digits/bits all $1$, the number itself has over $1318$ bits and is far from being a repunit.

In \eqref{eq:1}, the number $q$ represents the number of digits of $ns^{n}+1$ in base $b$, and the case where $q \le 2$
has been described in detail prior to Theorem~\ref{themaintheorem}.  So we may assume $q \ge 3$.  
For Theorem~\ref{theorem4}, we want to consider all integers $s \in [2,\currentbound]$.
Observe that \eqref{eq:1} implies 
\begin{equation}\label{universalqbound}
q < \dfrac{\log(ns^n)}{\log 2}.
\end{equation}
We will also consider all $n \le \max(s,N_s)$ with $q$ as above.

Next, we address
which $n > \max(s,N_s)$ need be considered as well, where such $n$ will depend on $s$.
The bound $\mathcal{N}_{s}$ in Theorem~\ref{theorem3} can be large and is the main difficulty in using the results in the prior 
sections directly.  For example, using $q_{1024} = 1024$, the bound $\mathcal N_{1024}$ based on \eqref{nsbound} is $\mathcal N_{1024} = 17496911$.  
Below we will show that in the second half of the argument of Theorem~\ref{theorem3}, we can restrict $q$ significantly further.

\begin{proposition}
If $n > \max(s,N_s)$, then $q < \log_2{s}+1.53$.
\end{proposition}

\begin{proof}
Assume that we have a solution to \eqref{eq:1} with $q \ge \log_2{s}+1.53$.
Since $n > \max(s,N_s) \ge N_{s}$, Theorem~\ref{lnlq} implies $q \le q_{s} = \max\{7,s \}$.  

We follow the argument for Theorem~\ref{theorem3}. 
In that argument, we considered $s = uv$, where $\gcd(u,v) = 1$, and $b = k u^{n}$ where $k \mid n$.  
There were two cases, where $u^{q-2} > v$ and where $u^{q-2} < v$.  
Following the argument for Theorem~\ref{theorem3}, we see that
for  the case $u^{q-2} > v$, we obtain a contradiction if $n^{1/n} \le u^{q-2}/v$;
and in the second case, where $u^{q-2} < v$, we obtain a contradiction if $(2n^{q})^{1/n} \le v/u^{q-2}$.
The rest of the argument for Theorem~\ref{theorem3} is based on the fact that the left-hand sides of these last two inequalities involving $n$
tend to $1$ for fixed $s$ as $n$ tends to infinity, and hence for large enough $n$ they hold.  
We next examine the cases $u^{q-2} > v$ and $u^{q-2} < v$ more thoroughly.

First, consider the case $u^{q-2}>v$. 
We have $u\ge 2$, so 
\[
\dfrac{u^{q-2}}{v} \ge \dfrac{2^{q-2}}{s/2} = \dfrac{2^{q-1}}{s}.
\]
By the assumption on $q$, this is greater than or equal to $2^{0.53} > 1.443$.
As the maximum value of $n^{1/n}$ for $n$ a positive integer is $3^{1/3} < 1.443$, we
deduce $n^{1/n} \le u^{q-2}/v$, implying the desired contradiction.

Next, consider $v>u^{q-2}$. If $u\ge 2$, then
\[
s > u^{q-1} \ge 2^{q-1} \ge 2^{0.53} s > s,
\] 
which is a contradiction.

If $u = 1$, then $b = k$ divides $n$ and $s = uv = v$.  
Furthermore, \eqref{eq:1} implies $\gcd( (n/b) s, b) = 1$.  
From Lemma~\ref{lemma2}, we have $b^q > s^n$.  
Thus, $b^{\max\{7,s\}} = b^{q_{s}} > s^{n}$.  
Since $b$ divides $n$ and $n > N_s$, we deduce
\begin{equation}\label{casetwoeq1}
n \ge b > s^{n/q} \ge s^{n/q_{s}}.
\end{equation}
If $q_{s} = 7$, then 
\[
n > (s^{1/q_{s}})^{n} \ge (2^{1/7})^{n} > 1.104^n,
\]
which does not hold for $n > 127$ and, hence, for $n > N_{s}$.  
For $q_{s} = s$, we use that $(s^{1/s})^{n} = n$ if $n = s$.  
For $s \ge 3$, the derivative of $(s^{1/s})^{x}$ as a function of $x$ 
is greater than $1$ for $x > s$ so that the Mean Value Theorem implies
$(s^{1/s})^{n} > n$ for all $n > s$.  
Since $n > \max(s,N_s) \ge s$, we see that we obtain a contradiction in \eqref{casetwoeq1} when $q_{s} = s$.
In the case that $q_{s} = s = 2$, a similar argument gives that \eqref{casetwoeq1} implies $n < 4$,
leading to a contradiction since $n > N_s \ge 127$.
\end{proof}

We describe now how we performed calculations to verify there are no solutions to \eqref{eq:1}, 
other than those described before Theorem~\ref{themaintheorem}, for all $s \in [2, \currentbound]$.  

For each such $s$, we consider all integers $n \in [1,\max(N_s,s)]$ and possibly more.

First, we rule out any new solutions to \eqref{eq:1} for $n \in [1,\max(N_s,s)]$. 
Let $f_{q}(b)$ be as defined in \eqref{definitionfq}.  
Observe that with $s$, $n$ and $q$ fixed, the value of a positive real number $b$ satisfying \eqref{eq:1} is uniquely determined. 
For each $s \in [2, \currentbound]$, $n \in [1,\max(N_s,s)]$ and $q \ge 3$ satisfying \eqref{universalqbound}, 
which does not give rise to a solution to \eqref{eq:1} described before Theorem~\ref{themaintheorem}, 
we form congruences 
\[
ns^n \equiv b f_q(b) \pmod{p}
\] 
for small primes $p$, until we find a $p$ where the congruence is impossible. 
We deduce that there are no integers $b$ satisfying \eqref{eq:1} for each such $s$, $n$ and $q$. 

Next, we consider $s \in [2, \currentbound]$ and $n > \max(N_s,s)$.  
Fix such an $s$, and let $N'_{s} = \max(N_s,s)+1$. 
For each divisor $u$ of $s$, we set $v = s/u$.  
We check whether $u^{q-2} > v$ or $u^{q-2} < v$.

In the case that $u^{q-2} > v$,
we want to consider $n \ge N'_{s}$ for which $n^{1/n} > u^{q-2}/v$, where $3 \le q < \log_2{s}+1.53$. 
For $x > e$, the function $x^{1/x}$ is decreasing to $1$ as $x \rightarrow \infty$.  
As a consequence, if ${N'_{s}}^{1/N'_{s}} \le u^{q-2}/v$, then $n^{1/n} \le u^{q-2}/v$ for all $n \ge N'_{s}$ and we need only consider $n < N'_{s}$. 
Since $u^{q-2} > v$, in the case that ${N'_{s}}^{1/N'_{s}} > u^{q-2}/v$, there is a real number $x_{1} = x_{1}(u,v,q) > N'_{s}$ satisfying $x_{1}^{1/x_{1}} = u^{q-2}/v$.
Then the number $x_{1}$ provides an upper bound on $n$ we need consider; more precisely, defining 
\[
n_{1} = n_{1}(u,v,q) =
\begin{cases}
N'_{s} &\text{if } {N'_{s}}^{1/N'_{s}} \le u^{q-2}/v, \\
x_{1} &\text{if } {N'_{s}}^{1/N'_{s}} > u^{q-2}/v, 
\end{cases}
\]
for a given $(u,v,q)$ satisfying $u^{q-2} > v$, we need only consider $n < n_{1}$.

In the case that $u^{q-2} < v$,
we want to consider $n \ge N'_{s}$ for which $(2n^{q})^{1/n} > v/u^{q-2}$,
where $3 \le q \le \log_2{s}+1.53$.   
For $x > e$, the function $(2x^{q})^{1/x}$  is decreasing to $1$ as $x \rightarrow \infty$. 
If $(2  {N'_{s}}^{q})^{1/N'_{s}} > v/u^{q-2}$, then 
there is a real number $x_{2} = x_{2}(u,v,q) > N'_{s}$ satisfying $(2 x_{2}^{q})^{1/x_{2}} = v/u^{q-2}$.
Defining
\[
n_{2} = n_{2}(u,v,q) =
\begin{cases}
N'_{s} &\text{if $(2  {N'_{s}}^{q})^{1/N'_{s}} \le v/u^{q-2}$}, \\
x_{2} &\text{if $(2  {N'_{s}}^{q})^{1/N'_{s}} > v/u^{q-2}$},
\end{cases}
\]
an analysis similar to the above shows that for a given $(u,v,q)$ satisfying $u^{q-2} < v$, we need only consider $n < n_{2}$.

We now complete the computation.  
For each $s \in [2,\currentbound]$, we considered each factorization $s = uv$ with $\gcd(u,v) = 1$ 
and considered each $q \in [2, \log_2{s}+1.53] \cap \mathbb Z$. 
In each case, we computed $n_{0} = n_{1}(u,v,q)$ or $n_{0} = n_{2}(u,v,q)$ depending on whether $u^{q-2} > v$ or $u^{q-2} < v$, respectively.  
We tested each integer $n \in (N'_{s},n_{0}]$ and found a small prime $p$ for which the congruence $ns^n \equiv q f_q(b) \pmod{p}$ has no solutions in $b$. 
This completes the proof of Theorem~\ref{theorem4}.

\bibliographystyle{alpha}
\bibliography{ucr}

\end{document}